\documentclass{tac}
\usepackage{amsmath,amssymb}
\usepackage{graphicx}
\usepackage{tikz-cd}
\usepackage{tikz}
\input diagxy
\usepackage{enumitem}

\author{Hongliang Lai and Bo Pang}

\thanks{We would like to thank the support of the National Natural Science Foundation of China (No. 12671089) 
}

\address{School of Mathematics, Sichuan University, Chengdu 610064, China.}

\title{Concept lattices of $Q$-relations via diagonal and back-diagonal categories }

\copyrightyear{2020}



\eaddress{hllai@scu.edu.cn\CR  bobiopen28@gmail.com}

\begin{document}

\def\oto{{\bfig\morphism<180,0>[\mkern-4mu`\mkern-4mu;]\place(86,0)[\circ]\efig}}
\def\rto{{\bfig\morphism<180,0>[\mkern-4mu`\mkern-4mu;]\place(82,0)[\mapstochar]\efig}}
\def\nto{{\bfig\morphism<180,0>[\mkern-4mu`\mkern-4mu;]\place(86,0)[\shortmid]\efig}}

\newcommand{\ra}{\rightarrow}
\newcommand{\QRel}{Q\text{-}\mathbf{Rel}}
\newcommand{\stk}{{\star_\mathrm{k}}}
 
\newcommand{\op}{\mathrm{op}}
\newcommand{\rmd}{\mathrm{d}}
\newcommand{\rmk}{\mathrm{k}}
\newcommand{\odt}{{\,{\scriptstyle\odot}\,}}
\newcommand{\sigt}{{\{\star\}}}
\newcommand{\bbA}{\mathbb{A}}
\newcommand{\bbB}{\mathbb{B}}
\newcommand{\bbC}{\mathbb{C}}
\newcommand{\bbQ}{\mathbb{Q}}
\newcommand{\CP}{\mathcal{P}}
 
\newcommand{\DQRel}{{\mathbf{D}(Q\text{-}\mathbf{Rel})}}
\newcommand{\BQRel}{{\mathbf{B}(Q\text{-}\mathbf{Rel})}}
\newcommand{\frB}{\mathfrak{B}}
\newcommand{\frD}{\mathfrak{D}}
\newcommand{\Sup}{\mathbf{Sup}}
\newcommand{\QSup}{{Q\text{-}\mathbf{Sup}}}

\maketitle

\begin{abstract}
Let $Q$ be a commutative unital quantale and $\QRel$ the ordered category of sets and $Q$-relations. We develop a categorical foundation for formal and property-oriented concept lattices of $Q$-relations via the diagonal category $\DQRel$ and back-diagonal category $\BQRel$. Our main results are: (1) the formal concept lattice construction is induced by a representable functor on $\BQRel$; (2) the property-oriented concept lattice construction corresponds to a representable functor on $\DQRel$; and (3) for regular $Q$-relations, the property-oriented concept lattice of their pointwise tensor product is naturally isomorphic to the tensor product of the respective property-oriented concept lattices.
\end{abstract}
\section{Introduction}
Formal concept analysis (FCA) and rough set theory (RST) are two complementary mathematical frameworks for extracting hierarchical conceptual structures from relational data, with wide applications in knowledge representation, data mining, and qualitative reasoning. Classical FCA produces formal concept lattices from formal contexts \cite{Ganter2024, Wille1982}, while RST yields property-oriented concept lattices \cite{Duntsch2002, Pawlak1982, Yao2004}, offering a dual paradigm for conceptual knowledge acquisition. To model fuzzy and graded relational information, these classical theories have been generalized to quantitative settings with truth values in a commutative unital quantale \(Q\), leading to the study of \(Q\)-contexts and their associated \(Q\)-valued concept lattices \cite{Belohlavek2002, Belohlavek2004, Georgescu2004, LaiZhang2009, Popescu2004}. This quantale-valued generalization recovers the classical binary framework when \(Q=\{0,1\}\) and provides a unified foundation for diverse quantitative concept analysis models.

The functorial behavior of concept lattices associated with \(Q\)-contexts is a fundamental topic in generalized formal concept analysis. For classical contexts, Ern\'{e} \cite{Erne1994,Erne2014} showed that the formal concept lattice construction defines a functor from the category \(\mathbf{CC}\) of contexts and conceptual morphisms to the category \(\mathbf{CLC}\) of complete lattices and complete lattice morphisms. Moreover, this functor is left adjoint to the inclusion \(\mathbf{CLC}\to\mathbf{CC}\), so that \(\mathbf{CLC}\) forms a reflective subcategory of \(\mathbf{CC}\).

Motivated by these classical results, this paper systematically investigates the functorial properties of both formal and property-oriented concept lattices over arbitrary \(Q\)-contexts.
Enriched categorical methods, including \(Q\)-categories\cite{Lawvere1973, Stubbe2005a, Stubbe2005b, Stubbe2007a} and (back-)diagonal constructions on quantaloids \cite{Shen2016,Shen2021, ShenTaoZhang2016,Stubbe2014}, have proven powerful for characterizing structural properties of quantitative relational systems. We construct the category \(\QSup\) of skeletal cocomplete \(Q\)-categories and left adjoint \(Q\)-functors, which provides a unified categorical environment for all formal and property-oriented concept lattices arising from \(Q\)-contexts. Based on the quantaloid \(\QRel\) of \(Q\)-relations, we establish an explicit functorial correspondence between categorical diagonal structures and the two standard quantitative concept analysis paradigms. Precisely, the FCA-based formal concept lattice construction for \(Q\)-contexts is induced by a representable functor \(\frB\) on the back-diagonal category \(\BQRel\), while the RST-based property-oriented concept lattice construction corresponds to a representable functor \(\frD\) on the diagonal category \(\DQRel\).

This functorial characterization clarifies how concept lattice assignments preserve morphisms and reveals essential structural differences between the two families of \(Q\)-valued concept lattices. In particular, if \(Q\) is a Girard quantale satisfying the double negation law, complementation of \(Q\)-relations induces an isomorphism between \(\BQRel\) and \(\DQRel\). This categorical isomorphism further yields a natural duality between the representable functors \(\frB\) and \(\frD\), turning both functors into categorical equivalences. This result recovers and extends the classical duality between fuzzy formal concept analysis and rough-set-based concept analysis \cite{LaiZhang2009,Yao2004} to a higher categorical level.

A central problem concerning the categorical structure of \(Q\)-contexts concerns the compatibility of two distinct tensor product constructions: the pointwise tensor product of \(Q\)-contexts used in concept analysis \cite{Erne1994, Ganter2024}, and the standard lattice tensor product in category theory \cite{Banaschewski1976, Garcia2017, JoyalTierney1984}. In general, the concept lattice of a pointwise tensor product of \(Q\)-relations fails to be isomorphic to the tensor product of the concept lattices of the original factors.

Building on the unified categorical interpretation of quantitative concept lattices, this paper resolves this tensor product compatibility problem. Exploiting the canonical monoidal closed structure of \(\QSup\), we prove that, for regular \(Q\)-relations, the property-oriented concept lattice of their pointwise tensor product is naturally isomorphic to the tensor product of their individual property-oriented concept lattices in \(\QSup\). Regular \(Q\)-relations are linked to constructively completely distributive (ccd) \(Q\)-categories via the representable functor $\frD$ over \(\DQRel\) \cite{LaiShen2018, Stubbe2007b}, implying that the full subcategory \(\QSup_{\mathrm{ccd}}\) of ccd \(Q\)-categories inherits a well-defined monoidal closed structure from \(\QSup\). Furthermore, when \(Q\) is Girard, the above compatibility result dualizes to formal concept lattices. We also provide explicit counterexamples to demonstrate that the regularity assumption on \(Q\)-relations is indispensable for the tensor product isomorphism to hold.

The remainder of the paper is organized as follows. Section 2 collects necessary preliminaries on commutative unital quantales, \(Q\)-relations, \(Q\)-categories, distributors, cocomplete \(Q\)-categories, and the monoidal closed structure of \(\QSup\). Section 3 studies the back-diagonal category \(\BQRel\), verifies its \(\QSup\)-enrichment, and identifies its representable functor $\frB$ with the formal concept lattice construction for \(Q\)-contexts. Section 4 develops the dual theory for the diagonal category \(\DQRel\), establishes its \(\QSup\)-enrichment, connects its representable functor $\frD$ to rough-set-based property-oriented concept lattices, discusses the dual isomorphism induced by Girard quantales, and illustrates the non-fullness of the representable functor via an explicit example. Section 5 investigates pointwise tensor products of regular \(Q\)-relations, characterizes the monoidal closed structure of idempotent \(Q\)-relations, proves the main tensor product isomorphism theorem for property-oriented concept lattices, and validates the necessity of the regularity condition by counterexamples.


\section{Preliminaries} \label{sec:prelim}

Let \(\Sup\) denote the category of complete lattices and arbitrary-join-preserving maps. A \emph{commutative unital quantale} \cite{Rosenthal1990} is precisely a commutative monoid object in \(\Sup\). Concretely, it is a complete lattice \(Q\) equipped with a commutative monoidal operation \(\&\) with unit \(\mathrm{k}\), such that \(\&\) distributes over arbitrary joins in each variable. Consequently, for every \(a\in Q\), the unary operation \(a\&(-)\colon Q\to Q\) admits a right adjoint \(a\ra(-)\colon Q\to Q\), called the \emph{residuum} of \(a\), which satisfies the standard adjoint equivalence
\begin{equation*}
a\& b\leq c \iff b\leq a\ra c \qquad (a,b,c\in Q).
\end{equation*}

A commutative unital quantale \(Q\) is \emph{integral} if its top element \(\top\) coincides with the monoidal unit \(\mathrm{k}\). Further, \(Q\) is a \emph{Girard quantale} whenever there exists an element \(\mathrm{d}\in Q\), termed the \emph{dualising element}, such that
\begin{equation*}
q=(q\ra \mathrm{d})\ra \mathrm{d}
\end{equation*}
holds for all \(q\in Q\). Throughout this paper, \(Q\) always denotes a fixed commutative unital quantale.

A \emph{\(Q\)-relation} \(\alpha\colon X\nto Y\) between two sets \(X\) and \(Y\) is a set-valued map \(\alpha\colon X\times Y\to Q\). The \emph{identity \(Q\)-relation} on a set \(X\) is defined pointwise by
\begin{equation*}
\delta_X(x,y)=
\begin{cases}
\mathrm{k}, & x=y,\\
\bot, & x\neq y.
\end{cases}
\end{equation*}

Any \(Q\)-relation of the form \(\alpha\colon X\nto\sigt\) or \(\alpha\colon\sigt\nto X\) can be canonically identified with a map \(\alpha\colon X\to Q\) via the set isomorphism \(X\times\sigt\cong X\). We therefore write \(\alpha(x)\) in place of \(\alpha(x,\star)\) or \(\alpha(\star,x)\) for simplicity.

The \emph{composition} of two composable \(Q\)-relations \(\alpha\colon X\nto Y\) and \(\beta\colon Y\nto Z\) is defined as
\begin{equation*}
(\beta\circ\alpha)(x,z)=\bigvee_{y\in Y}\beta(y,z)\&\alpha(x,y).
\end{equation*}

Since the quantale multiplication \(\&\) preserves arbitrary joins in each argument, relational composition is join-preserving in both variables:
\begin{equation*}
\beta\circ\Big(\bigvee_{i\in I}\alpha_i\Big)=\bigvee_{i\in I}\beta\circ\alpha_i,
\qquad
\Big(\bigvee_{i\in I}\beta_i\Big)\circ\alpha=\bigvee_{i\in I}\beta_i\circ\alpha.
\end{equation*}

Accordingly, the collection of all sets and \(Q\)-relations forms a quantaloid, that is, a $\Sup$-enriched category \cite{Rosenthal1996}, denoted \(\QRel\), whose hom-sets \(\QRel(X,Y)\) are ordered pointwise by the lattice order of \(Q\).

As a quantaloid, \(\QRel\) is equipped with internal hom operations \(\swarrow\) and \(\searrow\), which serve respectively as the right adjoints of post-composition and pre-composition:
\begin{equation*}
(-\circ\alpha)\dashv(-\swarrow\alpha),\qquad (\beta\circ-)\dashv(\beta\searrow-).
\end{equation*}

These adjunctions yield the following equivalent characterizations: for all \(\alpha\colon X\nto Y\), \(\beta\colon Y\nto Z\), \(\gamma\colon X\nto Z\),
\begin{equation*}
\beta\leq\gamma\swarrow\alpha \iff \beta\circ\alpha\leq\gamma \iff \alpha\leq\beta\searrow\gamma,
\end{equation*}
with explicit pointwise formulas
\begin{equation*}
(\gamma\swarrow\alpha)(y,z)=\bigwedge_{x\in X}\alpha(x,y)\ra\gamma(x,z),
\qquad
(\beta\searrow\gamma)(x,y)=\bigwedge_{z\in Z}\beta(y,z)\ra\gamma(x,z).
\end{equation*}

For any \(Q\)-relation \(\alpha\colon X\nto Y\), its \emph{opposite relation} \(\alpha^\op\colon Y\nto X\) is defined by \(\alpha^\op(y,x)=\alpha(x,y)\). Opposite relations satisfy the following natural identities with respect to composition and internal homs:
\begin{equation*}
(\beta\circ\alpha)^\op=\alpha^\op\circ\beta^\op,\quad
(\beta\searrow\gamma)^\op=\gamma^\op\swarrow\beta^\op,\quad
(\gamma\swarrow\alpha)^\op=\alpha^\op\searrow\gamma^\op.
\end{equation*}

\begin{proposition} \label{QRel self-dual}
The assignment \(\alpha\mapsto \alpha^\op\) defines an involution on \(\QRel\), so that the quantaloid \(\QRel\) is self-dual.
\end{proposition}

For each scalar \(q\in Q\) and each \(Q\)-relation \(\alpha\colon X\nto Y\), we define the \emph{scalar multiplication} and \emph{scalar implication} operations by
\begin{equation*}
q\&\alpha\colon X\nto Y,\quad (q\&\alpha)(x,y)=q\&\alpha(x,y),
\end{equation*}
\begin{equation*}
q\ra\alpha\colon X\nto Y,\quad (q\ra\alpha)(x,y)=q\ra\alpha(x,y).
\end{equation*}

The basic compatibility properties of these operations with relational composition and internal homs are summarized below.

\begin{proposition} \label{prop:scalar}
Let \(\alpha\colon X\nto Y\), \(\beta\colon Y\nto Z\), and \(\gamma\colon X\nto Z\) be arbitrary \(Q\)-relations. For all \(q\in Q\), the following identities hold:
\begin{enumerate}[label={\rm(\arabic*)}]
\item \((q\&\beta)\circ\alpha=\beta\circ(q\&\alpha)=q\&(\beta\circ\alpha)\);
\item \((q\&\beta)\searrow\gamma=q\ra(\beta\searrow\gamma)=\beta\searrow(q\ra\gamma)\);
\item \(\gamma\swarrow(q\&\alpha)=q\ra(\gamma\swarrow\alpha)=(q\ra\gamma)\swarrow\alpha\).
\end{enumerate}
\end{proposition}

The categories enriched over a quantale $Q$ are a special case of quantaloid-enriched categories \cite{Stubbe2005a,Stubbe2005b}.  
A \emph{\(Q\)-category} is a pair \(\mathbb{A}=(\mathbb{A}_0,\mathbb{A})\) consisting of a set \(\mathbb{A}_0\) of objects and a \(Q\)-relation \(\mathbb{A}\colon\mathbb{A}_0\nto\mathbb{A}_0\), called the hom functor of \(\mathbb{A}\), such that \(\delta_{\mathbb{A}_0}\leq\mathbb{A}\) and \(\mathbb{A}\circ\mathbb{A}\leq\mathbb{A}\). We write \(\mathbb{A}\) for the entire \(Q\)-category by standard convention.

The underlying preorder on \(\mathbb{A}\) is given by \(x\leq y\iff \mathrm{k}\leq\mathbb{A}(x,y)\). A \(Q\)-category \(\mathbb{A}\) is \emph{skeletal} if its underlying preorder is a partial order, i.e.,
\begin{equation*}
\mathrm{k}\leq\mathbb{A}(x,y)\text{ and }\mathrm{k}\leq\mathbb{A}(y,x)\iff x=y
\end{equation*}
for all \(x,y\in\mathbb{A}_0\).

A \emph{\(Q\)-functor} \(F\colon\mathbb{A}\to\mathbb{B}\) between \(Q\)-categories is an object map \(F\colon\mathbb{A}_0\to\mathbb{B}_0\) satisfying
\begin{equation*}
\mathbb{A}(x,y)\leq\mathbb{B}(Fx,Fy)
\end{equation*}
for all \(x,y\in\mathbb{A}_0\). All \(Q\)-categories and \(Q\)-functors constitute the locally ordered category \(Q\text{-}\mathbf{Cat}\).

A \emph{distributor} \(f\colon\mathbb{A}\oto\mathbb{B}\) is a \(Q\)-relation \(f\colon\mathbb{A}_0\nto\mathbb{B}_0\) satisfying the invariance condition \(\mathbb{B}\circ f\circ\mathbb{A}=f\). Distributors are closed under relational composition, forming the locally ordered category \(Q\text{-}\mathbf{Dist}\).

Every \(Q\)-functor \(F\colon\mathbb{A}\to\mathbb{B}\) induces two canonical distributors:
\begin{equation*}
F_{\natural}\colon\mathbb{A}\oto\mathbb{B},\quad F_{\natural}(x,y)=\mathbb{B}(Fx,y),
\qquad
F^{\natural}\colon\mathbb{B}\oto\mathbb{A},\quad F^{\natural}(y,x)=\mathbb{B}(y,Fx),
\end{equation*}
known as the graph and cograph of \(F\), respectively. A pair of \(Q\)-functors \(F\dashv G\) is \emph{adjoint} if \(F_{\natural}=G^{\natural}\), equivalently,
\begin{equation*}
\mathbb{B}(Fx,y)=\mathbb{A}(x,Gy)
\end{equation*}
for all admissible objects.

A set \(X\) equipped with its identity \(Q\)-relation \(\delta_X\) forms a \emph{discrete \(Q\)-category}. We write \(\star_\mathrm{k}\) for the discrete \(Q\)-category on a singleton set.

Distributors from \(\mathbb{A}\) to \(\star_\mathrm{k}\) are \emph{presheaves} on \(\mathbb{A}\), and they assemble into a \(Q\)-category \(\mathcal{P}\mathbb{A}\) with homs
\begin{equation*}
\mathcal{P}\mathbb{A}(\lambda,\mu)=\mu\swarrow\lambda
=\bigwedge_{x\in\mathbb{A}_0}\lambda(x)\ra\mu(x).
\end{equation*}

Dually, distributors from \(\star_\mathrm{k}\) to \(\mathbb{A}\) are \emph{copresheaves}, forming the \(Q\)-category \(\mathcal{P}^{\dagger}\mathbb{A}\) with
\begin{equation*}
\mathcal{P}^{\dagger}\mathbb{A}(\rho,\sigma)=\sigma\searrow\rho
=\bigwedge_{x\in\mathbb{A}_0}\sigma(x)\ra\rho(x).
\end{equation*}

The underlying order of \(\mathcal{P}\mathbb{A}\) coincides with the pointwise order of distributors, while that of \(\mathcal{P}^{\dagger}\mathbb{A}\) is dual:
\begin{equation*}
\rho\leq\sigma\text{ in } \mathcal{P}^{\dagger}\mathbb{A} \iff \sigma\leq\rho\text{ in } Q\text{-}\mathbf{Dist}(\star_\mathrm{k},\mathbb{A}).
\end{equation*}

The \emph{Yoneda embedding} and \emph{co-Yoneda embedding} are canonical \(Q\)-functors
\begin{equation*}
Y_{\mathbb{A}}\colon\mathbb{A}\to\mathcal{P}\mathbb{A},\quad Y_{\mathbb{A}}(x)=\mathbb{A}(-,x),
\qquad
Y_{\mathbb{A}}^\dagger\colon\mathbb{A}\to\mathcal{P}^\dagger\mathbb{A},\quad Y_{\mathbb{A}}^\dagger(x)=\mathbb{A}(x,-).
\end{equation*}

\begin{lemma}[Yoneda] \label{lem:yoneda}
For any \(Q\)-category \(\mathbb{A}\), all objects \(x\in\mathbb{A}_0\), presheaves \(\lambda\in\mathcal{P}\mathbb{A}\), and copresheaves \(\rho\in\mathcal{P}^{\dagger}\mathbb{A}\),
\begin{equation*}
\lambda(x)=\mathcal{P}\mathbb{A}(Y_{\mathbb{A}}x,\lambda),\qquad
\rho(x)=\mathcal{P}^{\dagger}\mathbb{A}(\rho,Y_{\mathbb{A}}^\dagger x).
\end{equation*}
\end{lemma}

A \(Q\)-category \(\mathbb{A}\) is \emph{cocomplete} \cite{Stubbe2005a} if its Yoneda embedding admits a left adjoint 
\[S_{\mathbb{A}}\dashv Y_{\mathbb{A}}\colon\bbA\to\CP\bbA.\]
 The left adjoint \(S_{\mathbb{A}}\) computes weighted colimits in \(\mathbb{A}\), and satisfies
\begin{equation*}
\mathbb{A}(S_{\mathbb{A}}\lambda,y)=\mathcal{P}\mathbb{A}(\lambda,Y_{\mathbb{A}}(y))
=\bigwedge_{z\in\mathbb{A}_0}\big(\lambda(z)\ra\mathbb{A}(z,y)\big).
\end{equation*}
We refer to \(S_{\mathbb{A}}\lambda\) as the \emph{supremum} of the presheaf \(\lambda\). 
A \(Q\)-category is cocomplete if and only if its opposite category is cocomplete. The cocompleteness of \(\mathbb{A}^\mathrm{op}\) is equivalent to the co-Yoneda embedding having a right adjoint.

The category \(\QSup\) is defined to consist of all skeletal cocomplete \(Q\)-categories and all left adjoint \(Q\)-functors between them. For each left adjoint morphism \(f\colon\mathbb{A}\to\mathbb{B}\) in \(\QSup\), we denote its right adjoint by \(f^\natural\colon\mathbb{B}\to\mathbb{A}\), so that \(f^\natural\dashv f\colon\mathbb{A}^\mathrm{op}\to\mathbb{B}^\mathrm{op}\).

\begin{proposition}
The correspondence \((\mathbb{A}\xrightarrow{f}\mathbb{B})\mapsto(\mathbb{B}^\mathrm{op}\xrightarrow{f^\natural}\mathbb{A}^\mathrm{op})\) yields an isomorphism between \(\QSup\) and its opposite category \(\QSup^\mathrm{op}\). Consequently, \(\QSup\) is self-dual.
\end{proposition}

Skeletal cocomplete \(Q\)-categories are equivalently characterized by \(Q\)-modules \cite{JoyalTierney1984,Stubbe2007a}. A \emph{\(Q\)-module} is a complete lattice \(M\) equipped with a scalar action \(Q\times M\to M\), written \(q\cdot x\), satisfying the following conditions:
\begin{enumerate}[label=(\arabic*)]
\item (unitality) $\rmk\cdot x=x$, 
\item (associativity) $q\cdot(p\cdot x)=(q\& p)\cdot x$, 
\item $\bigvee_{i\in I}( q_i\cdot x)=\Big(\bigvee_{i\in I}q_i\Big)\cdot x$, $\bigvee_{i\in I} (q\cdot x_i)=q\&\Big(\bigvee_{i\in I}x_i\Big)$.
 \end{enumerate}
 Each scalar action map \(q\cdot(-)\) on \(M\) has a right adjoint \((-)^q\), which induces a dual \(Q\)-module structure on \(M^\mathrm{op}\).

Every skeletal cocomplete \(Q\)-category \(\mathbb{A}\) admits a canonical \(Q\)-module structure on its underlying complete lattice, with action \(q\cdot x=S_{\mathbb{A}}(q\,\&\,Y_{\mathbb{A}}x)\). This yields the explicit supremum formula
\begin{equation*}
S_{\mathbb{A}}\lambda=\bigvee_{x\in\mathbb{A}_0}\lambda(x)\cdot x
\end{equation*}
for all presheaves \(\lambda\) on \(\mathbb{A}\). Conversely, any \(Q\)-module on a complete lattice $(A,\leq)$ induces a cocomplete \(Q\)-category structure
$$\bbA(x,y)=\bigvee\{q\in Q\mid q\cdot x\leq y\}.$$

\begin{example} \label{ex:discrete}
For any set \(X\), all maps \(\lambda\colon X\to Q\) serve simultaneously as presheaves and copresheaves on the discrete \(Q\)-category \(\delta_X\). The associated cocomplete \(Q\)-categories \(\mathcal{P}\delta_X\) and \(\mathcal{P}^\dagger\delta_X\) share the same underlying set \(Q^X\) but carry dual hom structures and orderings.  For each scalar $q\in Q$ and a presheaf $\lambda\in\CP\delta_X$, $q\cdot \lambda=q\&\lambda$, while for a copresheaf $\sigma\in\CP^\dagger\delta_X$, $q\cdot\sigma=q\ra\sigma$. 

In particular, \(\mathcal{P}\star_\mathrm{k}\) is a cocomplete \(Q\)-category with underlying set \(Q\), denoted \(\bbQ\), and \(\mathcal{P}^\dagger\star_\mathrm{k}=\bbQ^\mathrm{op}\).
\end{example}

\begin{proposition}[{\cite{Stubbe2007a}}] \label{prop:adjoint-char}
Let \(F\colon\mathbb{A}\to\mathbb{B}\) be a \(Q\)-functor between skeletal cocomplete \(Q\)-categories.
\begin{enumerate}[label={\rm(\arabic*)}]
\item \(F\) is a left adjoint if and only if it preserves scalar actions \(F(q\cdot x)=q\cdot F(x)\) and arbitrary joins.
\item \(F\) is a right adjoint if and only if it preserves dual scalar actions \(F(x^q)=F(x)^q\) and arbitrary meets.
\end{enumerate}
\end{proposition}

A morphism \(F\colon\mathbb{A}\times\mathbb{B}\to\mathbb{C}\) between skeletal cocomplete \(Q\)-categories is a \emph{\(\QSup\)-bimorphism} if it belongs to \(\QSup\) in each variable separately. The \emph{tensor product} \(\mathbb{A}\otimes\mathbb{B}\) is the universal codomain for such bimorphisms, endowing \(\QSup\) with a monoidal structure. The hom-category \(\QSup(\mathbb{A},\mathbb{B})\) is itself a skeletal cocomplete \(Q\)-category with pointwise hom formula and pointwise joins.

\begin{proposition}[{\cite{JoyalTierney1984}}] \label{prop:ccat-monoidal}
The category \(\QSup\) carries a symmetric monoidal closed structure with tensor product \(\otimes\) and unit \(\bbQ\). Specifically, for all skeletal cocomplete \(Q\)-categories \(\mathbb{A},\mathbb{B},\mathbb{C}\),
\begin{enumerate}[label={\rm(\arabic*)}]
\item \(\mathbb{A}\otimes\mathbb{B}\cong\mathbb{B}\otimes\mathbb{A}\);
\item \(\mathbb{A}\otimes\bbQ\cong\mathbb{A}\);
\item \(\QSup(\mathbb{A}\otimes\mathbb{B},\mathbb{C}) \cong \QSup(\mathbb{B},\QSup(\mathbb{A},\mathbb{C}))\);
\item \(\mathbb{A}\otimes\mathbb{B}\cong \QSup(\mathbb{A},\mathbb{B}^\mathrm{op})^\mathrm{op}\).
\end{enumerate}
\end{proposition}

\begin{proposition} \label{prop:qrel-enriched}
The quantaloid \(\QRel\) of \(Q\)-relations is a \(\QSup\)-enriched category.
\end{proposition}

\begin{proof}
As a quantaloid, \(\QRel\) is naturally \(\mathbf{Sup}\)-enriched: each hom-set \(\QRel(X,Y)\) constitutes a complete lattice, and relational composition preserves arbitrary joins in both variables. To establish \(\QSup\)-enrichment, it suffices to equip every hom-set with a compatible \(Q\)-module action.

We endow each \(\QRel(X,Y)\) with the standard scalar multiplication action \(q\cdot\alpha = q\,\&\,\alpha\). For any composable \(Q\)-relations \(\alpha\colon X\nto Y\), \(\beta\colon Y\nto Z\) and any scalar \(q\in Q\), direct computation yields
\[
\beta\circ(q\cdot\alpha) =\beta\circ(q\,\&\,\alpha) = q\,\&\,\left(\beta\circ\alpha\right) = q\cdot\left(\beta\circ\alpha\right).
\]
The symmetric identity \((q\cdot\beta)\circ\alpha = q\cdot(\beta\circ\alpha)\) holds by identical reasoning.

Therefore, relational composition commutes with the scalar \(Q\)-module action on both sides, which verifies that \(\QRel\) is a \(\QSup\)-enriched category.
\end{proof}


\section{Back diagonals and formal concepts of $Q$-relations}\label{sec:backdiag}
We first recall fundamental definitions and properties of formal concept lattices associated with $Q$-relations from \cite{Belohlavek2002,Shen2016}. A \emph{$Q$-context} is a triple $(X,Y,\alpha)$, where $\alpha\colon X\nto Y$ is a $Q$-relation between sets $X$ and $Y$. A pair $(\lambda,\sigma)\in \QRel(X,\sigt)\times \QRel(\sigt,Y)$ is said to be a \emph{formal concept} of the $Q$-context $(X,Y,\alpha)$ if it satisfies the reciprocal adjunction conditions
\[
\lambda=\sigma\searrow\alpha,\qquad \sigma=\alpha\swarrow\lambda.
\]

For any fixed $Q$-relation $\alpha\colon X\nto Y$, there exists a canonical adjunction
\[
(-)\searrow\alpha\vdash\alpha\swarrow(-)\colon\mathcal{P}\delta_X\to\mathcal{P}^\dagger\delta_Y
\]
between the skeletal cocomplete $Q$-categories $\mathcal{P}\delta_X$ and $\mathcal{P}^\dagger\delta_Y$. Within this adjunction, the two components of each formal concept are mutually determined and dual to one another. Precisely, the first component $\lambda$ is a fixed point satisfying $(\alpha\swarrow\lambda)\searrow\alpha=\lambda$, while dually, the second component $\sigma$ satisfies the fixed-point identity $\alpha\swarrow(\sigma\searrow\alpha)=\sigma$.

For any two formal concepts $(\lambda,\sigma)$ and $(\mu,\rho)$ of $(X,Y,\alpha)$, the hom functor equality
\[
\mathcal{P}\delta_X(\lambda,\mu)=\mathcal{P}^\dagger\delta_Y(\sigma,\rho)
\]
holds, which is equivalent to the relational identity $\mu\swarrow\lambda=\rho\searrow\sigma$.

All formal concepts of a $Q$-context $(X,Y,\alpha)$ constitute a skeletal cocomplete $Q$-category $\frB(X,Y,\alpha)$, referred to as the \emph{formal concept lattice} of the $Q$-context. This concept lattice can be embedded as a full sub-$Q$-category into both $\mathcal{P}\delta_X$ and $\mathcal{P}^\dagger\delta_Y$. By endowing the hom functor with
\[
\frB(X,Y,\alpha)((\lambda,\sigma),(\mu,\rho))=\mu\swarrow\lambda=\rho\searrow\sigma,
\]
$\frB(X,Y,\alpha)$ forms a well-defined skeletal cocomplete $Q$-category, as verified in \cite{Belohlavek2002,LaiZhang2009}.

Our subsequent aim is to lift the object assignment $$(X,Y,\alpha)\mapsto\frB(X,Y,\alpha)$$ to a functor between suitable categories of $Q$-contexts. This lifting requires a rigorous notion of morphisms between $Q$-contexts, for which the back diagonal morphisms of $Q$-relations provide a natural and well-suited framework.

Back diagonals exist for any quantaloid (see \cite{ShenTaoZhang2016}), yet in what follows we focus exclusively on back diagonals in \(\QRel\). The category of back diagonals \(\BQRel\) within the quantaloid \(\QRel\) is defined below:
\begin{itemize}
\item \textbf{Objects}: All $Q$-relations in $\QRel$.
\item \textbf{Morphisms}: A morphism from a $Q$-relation $\alpha\colon X\nto Y$ to another $Q$-relation $\beta\colon X'\nto Y'$ is a $Q$-relation $f\colon X'\nto Y$ satisfying
\[
\alpha\swarrow(f\searrow\alpha)=f=(\beta\swarrow f)\searrow\beta.
\]
Or equivalently, 
\[
\exists u\colon X\nto X',\ \exists v\colon Y\nto Y',\quad \alpha\swarrow u=f=v\searrow\beta.
\]
\item \textbf{Composition}: For composable morphisms $f\colon\alpha\to\beta$ and $g\colon\beta\to\gamma$, their composite is given by
\[
g\bullet f:=(\beta\swarrow f)\searrow g=f\swarrow(g\searrow\beta).
\]
\end{itemize}

The category $\mathbf{B}(Q\text{-}\mathbf{Rel})$ inherits a quantaloid structure, where the partial order on each hom-set $\mathbf{B}(Q\text{-}\mathbf{Rel})(\alpha,\beta)$ is the opposite of the local order induced from $\QRel$. As an immediate consequence of Proposition \ref{QRel self-dual}, we establish the self-duality of $\mathbf{B}(Q\text{-}\mathbf{Rel})$.

\begin{proposition}
The object-wise assignment $\alpha\mapsto\alpha^\op$ extends to a category isomorphism between $\BQRel$ and its opposite category $(\BQRel)^\op$. Consequently, $\BQRel$ is self-dual.
\end{proposition}

We further verify that the back diagonal category admits enriched structure over the category of cocomplete $Q$-categories.

\begin{proposition} \label{prop:b-enriched}
The category $\mathbf{B}(Q\text{-}\mathbf{Rel})$ is $\QSup$-enriched.
\end{proposition}

\begin{proof}
As a quantaloid, $\mathbf{B}(Q\text{-}\mathbf{Rel})$ is inherently $\mathbf{Sup}$-enriched: every hom-set forms a complete lattice, and the composition operation preserves arbitrary local joins in both arguments. To upgrade this enrichment to $\QSup$-enrichment, it suffices to equip each hom-set with a compatible $Q$-action and verify that composition commutes with this action in both variables.

First, we define a canonical $Q$-action on the hom-sets of $\mathbf{B}(Q\text{-}\mathbf{Rel})$. For any morphism $f\in\mathbf{B}(Q\text{-}\mathbf{Rel})(\alpha,\beta)$ and any scalar $q\in Q$, we set
\[
q\cdot f=q\ra f,
\]
where $\ra$ denotes the scalar implication operation on $Q$-relations. We now confirm that $q\cdot f$ is indeed a morphism in $\mathbf{B}(Q\text{-}\mathbf{Rel})(\alpha,\beta)$. Since there exist $Q$-relations $u,v$ such that $f=\alpha\swarrow u=v\searrow\beta$, applying the compatibility between scalar implication and meets of $Q$-relations (Proposition \ref{prop:scalar}), we compute
\[
\alpha\swarrow(q\& u)=q\ra(\alpha\swarrow u)=q\cdot f
=q\ra(v\searrow\beta)=(q\& v)\searrow\beta.
\]
This shows that $q\cdot f$ is a back diagonal morphism. Thus, the assignment $(q,f)\mapsto q\cdot f$ yields a well-defined $Q$-action on each hom-set of $\mathbf{B}(Q\text{-}\mathbf{Rel})$, endowing every hom-set with the structure of a cocomplete $Q$-category.

Next, we prove that the composition operation $\bullet$ preserves the $Q$-action in both arguments, making it a $\QSup$-bimorphism. Consider composable morphisms $f\in\mathbf{B}(Q\text{-}\mathbf{Rel})(\alpha,\beta)$ and $g\in\mathbf{B}(Q\text{-}\mathbf{Rel})(\beta,\gamma)$. Direct calculation gives
\[
\begin{aligned}
(q\cdot g)\bullet f
&=(\beta\swarrow f)\searrow(q\cdot g)\\
&=(\beta\swarrow f)\searrow(q\ra g)\\
&=(q\&(\beta\swarrow f))\searrow g\\
&=q\ra((\beta\swarrow f)\searrow g)\\
&=q\cdot(g\bullet f).
\end{aligned}
\]
By duality, the identity $g\bullet(q\cdot f)=q\cdot(g\bullet f)$ also holds. These identities confirm that composition commutes with the scalar $Q$-action on both sides. Therefore, $\mathbf{B}(Q\text{-}\mathbf{Rel})$ is a $\QSup$-enriched category as desired.
\end{proof}

Let $\tau\colon\sigt\nto Q$ denote the canonical $Q$-relation defined by $\tau(\star,q)=q$ for all $q\in Q$. Utilizing the $\QSup$-enriched structure of $\mathbf{B}(Q\text{-}\mathbf{Rel})$, we define the representable functor
\[
\frB:=\mathbf{B}(Q\text{-}\mathbf{Rel})(-,\tau)\colon\BQRel^\op\to\QSup.
\]
Explicitly, for each $Q$-relation $\alpha\colon X\nto Y$, the object action is $\mathfrak{B}\alpha=\mathbf{B}(Q\text{-}\mathbf{Rel})(\alpha,\tau)$, and for each morphism $f\colon\alpha\to\beta$ in $\mathbf{B}(Q\text{-}\mathbf{Rel})$, the morphism action is given by $\mathfrak{B}f(\sigma)=\sigma\bullet f$.

\begin{theorem}
For any $Q$-context $(X,Y,\alpha)$, there exists a canonical isomorphism of $Q$-categories
\[
\frB(X,Y,\alpha)\cong\frB\alpha.
\]
Consequently, the formal concept lattice construction $(X,Y,\alpha)\mapsto\frB(X,Y,\alpha)$ is functorial over the category $\BQRel^\op$.
\end{theorem}

\begin{proof}
Let $(\lambda,\sigma)$ be any formal concept of $(X,Y,\alpha)$. Direct computation yields
\[
\begin{aligned}
((\tau\swarrow\sigma)\searrow\tau)(x)
&=\bigwedge_{q\in Q}\big((\sigma(x)\ra\tau(q))\ra\tau(q)\big)\\
&=\bigwedge_{q\in Q}\big((\sigma(x)\ra q)\ra q\big)\\
&=\sigma(x)
\end{aligned}
\]
for all $x\in X$, which establishes the identity $(\tau\swarrow\sigma)\searrow\tau=\sigma$. By the formal concept conditions $\sigma=\alpha\swarrow\lambda$ and $\lambda=\sigma\searrow\alpha$, we obtain $\sigma=\alpha\swarrow(\sigma\searrow\alpha)$. Thus, $\sigma\in\frB\alpha$.

Conversely, for any $\sigma\in\frB\alpha$, define $\lambda=\sigma\searrow\alpha$. The fixed-point condition for $\frB\alpha$ guarantees $\sigma=\alpha\swarrow\lambda$, so that $(\lambda,\sigma)$ is a formal concept of $(X,Y,\alpha)$. This yields a bijection between the objects of $\frB(X,Y,\alpha)$ and $\frB\alpha$.

For all \(\sigma,\rho\in\frB\alpha\), the hom function of \(\frB\alpha\) is given by
\[
\frB\alpha(\sigma,\rho)
=\bigvee\{q\in Q\mid q\ra\sigma\geq\rho\}
=\rho\searrow\sigma,
\]
which coincides with the hom function of the \(Q\)-category \(\frB(X,Y,\alpha)\).

The object bijection together with identical hom functions induces a canonical $Q$-category isomorphism $\frB(X,Y,\alpha)\cong\frB\alpha$. The functoriality of the formal concept lattice construction follows immediately from the functoriality of $\frB$.
\end{proof}

\begin{proposition}[\cite{ShenTaoZhang2016}] \label{prop:b-equivalence}
The representable functor $\mathfrak{B}\colon\mathbf{B}(Q\text{-}\mathbf{Rel})^{\mathrm{op}}\to\QSup$ is a category equivalence.
\end{proposition}

\section{Diagonals and property oriented concepts of $Q$-relations} \label{sec:diag}

This section develops the dual theory for property-oriented concepts of $Q$-contexts, based on the diagonal category of $Q$-relations. Let $(X,Y,\alpha)$ be a $Q$-context. A pair $(\lambda,\mu)\in \QRel(Y,\sigt)\times \QRel(X,\sigt)$ is called a \emph{property-oriented concept} \cite{Duntsch2002,Popescu2004,Yao2004} of $(X,Y,\alpha)$ if it satisfies the system
\[
\lambda=\mu\swarrow\alpha,\qquad \mu=\lambda\circ\alpha.
\]

Any given $Q$-relation $\alpha\colon X\nto Y$ induces a canonical adjunction
\[
(-)\swarrow\alpha\vdash(-)\circ\alpha\colon\mathcal{P}\delta_Y\to\mathcal{P}\delta_X
\]
between skeletal cocomplete $Q$-categories $\mathcal{P}\delta_Y$ and $\mathcal{P}\delta_X$. Similar to the case of formal concepts, the two components of each property-oriented concept are mutually determined via this adjunction. Specifically, $\lambda$ is a fixed point of the $Q$-functor $$((-)\circ\alpha)\swarrow\alpha\colon\mathcal{P}\delta_Y\to\mathcal{P}\delta_Y,$$
 while $\mu$ is a fixed point of the dual $Q$-functor $$((-)\swarrow\alpha)\circ\alpha\colon\mathcal{P}\delta_X\to\mathcal{P}\delta_X.$$

For any two property-oriented concepts $(\lambda,\mu)$ and $(\lambda',\mu')$ of $(X,Y,\alpha)$, the hom functor compatibility condition
\[
\mathcal{P}\delta_X(\lambda,\lambda')=\mathcal{P}\delta_Y(\mu,\mu')
\]
holds, which translates to the relational identity $\lambda'\swarrow\lambda=\mu'\swarrow\mu$.

All property-oriented concepts of a $Q$-context $(X,Y,\alpha)$ form a skeletal cocomplete $Q$-category $\mathfrak{P}(X,Y,\alpha)$, termed the \emph{property-oriented concept lattice} of the $Q$-context. This lattice can be embedded as a full sub-$Q$-category into both $\mathcal{P}\delta_X$ and $\mathcal{P}\delta_Y$. Equipping the hom functor with
\[
\mathfrak{P}(X,Y,\alpha)((\lambda,\mu),(\lambda',\mu'))=\lambda'\swarrow\lambda=\mu'\swarrow\mu,
\]
endows $\mathfrak{P}(X,Y,\alpha)$ with a well-defined skeletal cocomplete $Q$-category structure \cite{LaiZhang2009}.

Analogously to the formal-concept setting, we lift the object assignment  $$(X,Y,\alpha)\mapsto\mathfrak{P}(X,Y,\alpha)$$ to a functor by means of the category of diagonal morphisms for $Q$-relations. Diagonals exist for any quantaloid (see \cite{Stubbe2014}), yet in what follows we focus exclusively on diagonals in \(\QRel\). The category of diagonals \(\DQRel\) within the quantaloid \(\QRel\) is defined below:
\begin{itemize}
\item \textbf{Objects}: All $Q$-relations in $\QRel$.
\item \textbf{Morphisms}: A morphism from $\alpha\colon X\nto Y$ to $\beta\colon X'\nto Y'$ is a $Q$-relation $f\colon X\nto Y'$ satisfying
\[
(f\swarrow\alpha)\circ\alpha=f=\beta\circ(\beta\searrow f).
\]
Equivalently, $f$ admits the factorization
\[
\exists u\colon X\nto X',\ \exists v\colon Y\nto Y',\quad v\circ\alpha=f=\beta\circ u.
\]
\item \textbf{Composition}: For composable morphisms $f\colon\alpha\to\beta$ and $g\colon\beta\to\gamma$, their composite is defined as
\[
g\odt f:=g\circ(\beta\searrow f)=(g\swarrow\beta)\circ f.
\]
\end{itemize}

The category \(\DQRel\) is a quantaloid in which the partial order on each hom-set \(\DQRel(\alpha,\beta)\) coincides with the local order inherited from \(\QRel\). Moreover, \(\QRel\) is a full sub-quantaloid of \(\DQRel\), since every \(Q\)-relation \(\alpha\colon X\nto Y\) canonically determines a diagonal morphism from \(\delta_X\) to \(\delta_Y\). The assignment \(\alpha\mapsto\alpha^\op\) extends to an isomorphism between \(\DQRel\) and \(\DQRel^\op\); hence \(\DQRel\) is self-dual. 

We now establish the $\QSup$-enriched structure of the diagonal category, mirroring the result for the back diagonal category.

\begin{proposition} \label{prop:d-enriched}
The category $\mathbf{D}(Q\text{-}\mathbf{Rel})$ is $\QSup$-enriched.
\end{proposition}

\begin{proof}
As a quantaloid, $\mathbf{D}(Q\text{-}\mathbf{Rel})$ is $\mathbf{Sup}$-enriched, with complete lattice hom-sets and bilaterally join-preserving composition. To establish $\QSup$-enrichment, we verify the existence of a compatible $Q$-action on all hom-sets and confirm that composition preserves this action bilaterally.

For any morphism $f\in\mathbf{D}(Q\text{-}\mathbf{Rel})(\alpha,\beta)$ and scalar $q\in Q$, we define the $Q$-action by
\[
q\cdot f=q\& f,
\]
i.e., the scalar multiplication of $q$ and $f$. We verify that $q\cdot f$ remains a diagonal morphism. Using the morphism factorization property, there exist $Q$-relations $u,v$ such that $f=v\circ\alpha=\beta\circ u$. By Proposition \ref{prop:scalar}(1), we compute
\[
(q\& v)\circ\alpha=q\&(v\circ\alpha)=q\& f=q\&(\beta\circ u)=\beta\circ(q\& u).
\]
This factorization confirms that $q\cdot f$ satisfies the diagonal morphism condition, so the assignment $(q,f)\mapsto q\cdot f$ defines a valid $Q$-action on each hom-set of $\mathbf{D}(Q\text{-}\mathbf{Rel})$, equipping each hom-set with a cocomplete $Q$-category structure.

Next, we verify the compatibility of composition with the $Q$-action. For composable morphisms $f\in\mathbf{D}(Q\text{-}\mathbf{Rel})(\alpha,\beta)$ and $g\in\mathbf{D}(Q\text{-}\mathbf{Rel})(\beta,\gamma)$, direct calculation yields
\[
(q\cdot g)\odt f=(q\& g)\circ(\beta\searrow f)
=q\&(g\circ(\beta\searrow f))=q\cdot(g\odt f).
\]
The dual identity
\[
g\odt(q\cdot f)=(g\swarrow\beta)\circ(q\& f)
=q\&((g\swarrow\beta)\circ f)=q\cdot(g\odt f)
\]
holds symmetrically. These equalities demonstrate that the diagonal composition operation commutes with the scalar $Q$-action in both arguments. Therefore, $\mathbf{D}(Q\text{-}\mathbf{Rel})$ is a $\QSup$-enriched category as desired.
\end{proof}

Leveraging the $\QSup$-enriched structure of $\mathbf{D}(Q\text{-}\mathbf{Rel})$, we define the representable functor
\[
\mathfrak{D}:=\mathbf{D}(Q\text{-}\mathbf{Rel})(-,\star_\mathrm{k})\colon\mathbf{D}(\QRel)^\op\to\QSup,
\]
which sends each $Q$-relation $\alpha\colon X\nto Y$ to the cocomplete $Q$-category $\mathfrak{D}\alpha$, and each morphism $f\in\mathbf{D}(Q\text{-}\mathbf{Rel})(\beta,\alpha)$ to the left adjoint $\mathfrak{D}f\colon\mathfrak{D}\alpha\to\mathfrak{D}\beta$ given by $\mathfrak{D}f(\lambda)=\lambda\odt f$. 

\begin{theorem}
For any $Q$-context $(X,Y,\alpha)$, there exists a canonical isomorphism of $Q$-categories:
\[
\mathfrak{P}(X,Y,\alpha)\cong\frD\alpha.
\]
Accordingly, the property-oriented concept lattice construction $(X,Y,\alpha)\mapsto\mathfrak{P}(X,Y,\alpha)$ is functorial over the category $\DQRel^\op$.
\end{theorem}

\begin{proof}
We first verify the object-wise correspondence between $\mathfrak{P}(X,Y,\alpha)$ and $\frD\alpha$. For any $\mu\in\frD\alpha$, define $\lambda=\mu\swarrow\alpha$. The fixed-point condition immediately implies $\mu=\lambda\circ\alpha$, so that the pair $(\lambda,\mu)$ is exactly a property-oriented concept of the $Q$-context $(X,Y,\alpha)$.

Conversely, suppose $(\lambda,\mu)$ is an arbitrary property-oriented concept of $(X,Y,\alpha)$. By the definition of property-oriented concepts, the equalities $\mu=\lambda\circ\alpha$ and $\lambda=\mu\swarrow\alpha$ hold simultaneously. Substitution yields $\mu=(\mu\swarrow\alpha)\circ\alpha$, which ensures $\mu\in\frD\alpha$. This establishes a bijection between the objects of $\mathfrak{P}(X,Y,\alpha)$ and $\frD\alpha$.

The hom function of $\frD\alpha$ is characterized by
\[
\frD\alpha(\mu,\mu')=\bigvee\{q\in Q\mid q\&\mu\leq\mu'\}=\mu'\swarrow\mu,\quad \forall\mu,\mu'\in\frD\alpha. 
\]
This perfectly matches the hom function of the $Q$-category $\mathfrak{P}(X,Y,\alpha)$.

The object bijection together with the coincidence of hom functors induces a canonical $Q$-category isomorphism $\mathfrak{P}(X,Y,\alpha)\cong\frD\alpha$. Furthermore, since $\frD$ is functorial, the assignment $(X,Y,\alpha)\mapsto\mathfrak{P}(X,Y,\alpha)$ on objects inherits the functoriality on the category $\DQRel^\op$.
\end{proof}

When the base quantale $Q$ is a Girard quantale, the diagonal and back diagonal categories become categorically isomorphic, as established in \cite{Shen2021}. For a $Q$-relation $\alpha\colon X\nto Y$, define its \emph{complement} as the $Q$-relation $\neg\alpha\colon Y\nto X$ with
\[
\neg\alpha(y,x)=\alpha(x,y)\ra\mathrm{d},
\]
where $\mathrm{d}$ is the dualizing element of the Girard quantale $Q$. We now verify the morphism-wise correspondence induced by complementation:
\begin{align*}
f\in\mathbf{D}(Q\text{-}\mathbf{Rel})(\alpha,\beta)
&\iff (f\swarrow\alpha)\circ\alpha=f=\beta\circ(\beta\searrow f)\\
&\iff \neg((f\swarrow\alpha)\circ\alpha)=\neg f=\neg(\beta\circ(\beta\searrow f))\\
&\iff (\neg\alpha)\swarrow(f\swarrow\alpha)=\neg f=(\beta\searrow f)\searrow(\neg\alpha)\\
&\iff (\neg\alpha)\swarrow((\neg f)\searrow(\neg\alpha))=\neg f=((\neg\beta)\swarrow(\neg f))\searrow(\neg\alpha)\\
&\iff \neg f\in\mathbf{B}(Q\text{-}\mathbf{Rel})(\neg\alpha,\neg\beta).
\end{align*}
For composable morphisms $f\colon\alpha\to\beta$ and $g\colon\beta\to\gamma$ in $\mathbf{D}(\QRel)$, the complementation operation further commutes with composition:
\begin{align*}
\neg(g\odt f)&=\neg(g\circ(\beta\searrow f))\\
&=(\beta\searrow f)\searrow(\neg g)\\
&=((\neg\beta)\swarrow(\neg f))\searrow(\neg g)\\
&=(\neg g)\bullet(\neg f).
\end{align*}
These computations confirm that the complementation assignment $\alpha\mapsto\neg\alpha$ defines a category isomorphism
\[
\neg\colon \DQRel\to\BQRel.
\]

\begin{lemma} \label{lem:girard-b}
Let $Q$ be a Girard quantale, and let $\star_\mathrm{d}$ denote the $Q$-relation on the singleton set $\{\star\}$ with $\star_\mathrm{d}(\star,\star)=\mathrm{d}$. Then the functor $\mathfrak{B}$ admits the alternative representation $\mathfrak{B}=\mathbf{B}(Q\text{-}\mathbf{Rel})(-,\star_\mathrm{d})$. Moreover, the following diagram commutes:
\[
\begin{tikzcd}
\mathbf{D}(Q\text{-}\mathbf{Rel})^{\mathrm{op}}
\arrow{rr}{\neg} \arrow{dr}[swap]{\mathfrak{D}} &&
\mathbf{B}(Q\text{-}\mathbf{Rel})^{\mathrm{op}} \arrow{dl}{\mathfrak{B}} \\
& \QSup &
\end{tikzcd}
\]
In this case, $\frD:\DQRel^\op\to\QSup$ is an equivalence. 
\end{lemma}

\begin{proof}
For any $Q$-relation $\alpha\colon X\nto Y$, a $Q$-relation $\sigma\colon\{\star\}\nto Y$ lies in $\mathbf{B}(\QRel)(\alpha,\star_\mathrm{d})$ if and only if $\alpha\swarrow(\sigma\searrow\alpha)=\sigma$, since the identity $(\star_\mathrm{d}\swarrow\sigma)\searrow\star_\mathrm{d}=\sigma$ holds automatically for Girard quantales. This condition is exactly the defining property of objects in $\mathfrak{B}\alpha$, yielding $\mathfrak{B}\alpha=\BQRel(\alpha,\star_\mathrm{d})$ for all $Q$-relations $\alpha$. The commutative diagram follows immediately.
\end{proof}

Provided that the base quantale $Q$ is integral, the functor $\frD\colon\DQRel^\op\to\QSup$ restricts to a categorical equivalence if and only if $Q$ is a Girard quantale.

\begin{proposition}[\cite{LaiZhang2009}] \label{prop:ess-surj}
Let \(Q\) be an integral quantale. The representable functor
$$
\frD\colon \DQRel^{\mathrm{op}}\to \QSup
$$
is essentially surjective if and only if \(Q\) is a Girard quantale. Moreover, \(\frD\) is an equivalence precisely when \(Q\) is Girard.
\end{proposition}

For general non-Girard quantales, the functor $\mathfrak{D}$ fails to be an equivalence, though it retains faithful behavior. We formally establish the faithfulness of $\mathfrak{D}$ below.

\begin{proposition} \label{prop:faithful}
The discrete singleton $Q$-category $\star_\mathrm{k}$ serves as both a separator and a coseparator for $\mathbf{D}(\QRel)$. As a consequence, the functor $\mathfrak{D}\colon\mathbf{D}(Q\text{-}\mathbf{Rel})^\op\to\QSup$ is faithful.
\end{proposition}

\begin{proof}
Let $f,g\colon\beta\to\alpha$ be distinct parallel morphisms in $\mathbf{D}(\QRel)$, where $\beta\colon X'\nto Y'$ and $\alpha\colon X\nto Y$. By distinctness, there exists an element $y_0\in Y$ such that $f(-,y_0)\neq g(-,y_0)$. Note that the $Q$-relation $\alpha(-,y_0)\colon X\nto\{\star\}$ is a morphism in $\mathbf{D}(\QRel)(\alpha,\star_\mathrm{k})$, as it satisfies $\alpha(-,y_0)=\delta_Y(-,y_0)\circ\alpha$. Direct computation gives
\[
\alpha(-,y_0)\odt f=\alpha(-,y_0)\circ(\alpha\searrow f)=\bigvee_{x\in X}\alpha(x,y_0)\&(\alpha\searrow f)(-,x)=f(-,y_0),
\]
and similarly
\[
\alpha(-,y_0)\odt g=g(-,y_0).
\]
The inequality $f(-,y_0)\neq g(-,y_0)$ implies $\alpha(-,y_0)\odt f\neq\alpha(-,y_0)\odt g$, so $\star_\mathrm{k}$ is a coseparator for $\mathbf{D}(\QRel)$. Dually, one verifies that $\star_\mathrm{k}$ is also a separator. The faithfulness of $\mathfrak{D}$ follows from the fact that $\star_\mathrm{k}$ is a coseparating object.
\end{proof}

We conclude this section with an explicit example demonstrating that $\mathfrak{D}$ is not full in general, even though it is always faithful.

\begin{example} \label{ex:not-full}
Let $Q=([0,\infty],\geq,+)$ be the Lawvere quantale, with monoidal operation $\&=+$ and reversed natural order. Consider the two $Q$-relations on the singleton set $\{\star\}$ defined by
\[
\alpha\colon\sigt\nto\sigt,\quad \alpha(\star,\star)=0,\qquad \beta\colon\sigt\nto\sigt,\quad \beta(\star,\star)=b,\quad 0<b<\infty.
\]
Direct computation yields
\[
\mathfrak{D}\alpha=[0,\infty],\qquad \mathfrak{D}\beta=[b,\infty],\qquad \mathbf{D}(Q\text{-}\mathbf{Rel})(\alpha,\beta)=[b,\infty].
\]
For any morphism $r\in\mathbf{D}(\QRel)(\alpha,\beta)$ and any object $\sigma\in\mathfrak{D}\beta$, the functor action is
\[
\mathfrak{D}r(\sigma)=\sigma\odt r=\sigma-b+r\geq b.
\]
The map $\sigma\mapsto\sigma-b$ defines a $\QSup$-isomorphism from $\mathfrak{D}\beta$ to $\mathfrak{D}\alpha$, yet this isomorphism cannot be realized as the image of any morphism in $\mathbf{D}(\QRel)$ under $\mathfrak{D}$. This proves that the functor $\mathfrak{D}$ is not full for the Lawvere quantale $Q$.
\end{example}

    
\section{Pointwise tensor products of $Q$-relations and tensor products of cocomplete $Q$-categories}
 The functor $\mathfrak{B}\colon\BQRel^\op\to\QSup$ induces a categorical equivalence. Moreover, if $Q$ is a Girard quantale, the three categories $\QSup$, $\BQRel^\op$, and $\DQRel^\op$ are pairwise equivalent. It is well known that $\QSup$ admits a canonical symmetric monoidal closed structure with respect to its standard tensor product $\otimes$. In contrast, arbitrary pairs of $Q$-relations admit a natural pointwise tensor product $*$ defined componentwise. This section systematically explores the intrinsic relationship between these two tensor product constructions and establishes structural compatibility results for regular $Q$-relations and idempotent $Q$-relations, as well as for constructively completely distributive cocomplete $Q$-categories.

Given two $Q$-relations $\alpha\colon X\nto Y$ and $\beta\colon X'\nto Y'$, we define their \emph{pointwise tensor product} to be the $Q$-relation $\alpha*\beta\colon X\times X'\nto Y\times Y'$ whose values are specified pointwise by
\[
(\alpha*\beta)\bigl((x,x'),(y,y')\bigr)
=\alpha(x,y)\mathbin{\&}\beta(x',y').
\]

This pointwise tensor product is compatible with the composition of $Q$-relations in $\QRel$. Such compatibility further extends to the diagonal composition of morphisms in $\DQRel$, as formalized in the following results.

\begin{proposition}\label{prop:tensor-compat}
Let $\alpha\colon X\nto Y$, $\beta\colon Y\nto Z$ and $\alpha'\colon X'\nto Y'$, $\beta'\colon Y'\nto Z'$ be composable $Q$-relations. The pointwise tensor product commutes with relational composition, i.e.,
\[
(\beta\circ\alpha)*(\beta'\circ\alpha')
=(\beta*\beta')\circ(\alpha*\alpha').
\]
\end{proposition}

\begin{proof}
This identity follows from direct pointwise computation, together with the fact that the commutative monoidal operation $\&$ in $Q$ distributes over arbitrary suprema. For all $(x,x')\in X\times X'$ and $(z,z')\in Z\times Z'$,
\[
\begin{aligned}
&\bigl((\beta*\beta')\circ(\alpha*\alpha')\bigr)\bigl((x,x'),(z,z')\bigr)\\
&\quad=\bigvee_{y\in Y,\,y'\in Y'}
\bigl(\beta(y,z)\mathbin{\&}\beta'(y',z')\bigr)
\mathbin{\&}\bigl(\alpha(x,y)\mathbin{\&}\alpha'(x',y')\bigr)\\
&\quad=\Bigl(\bigvee_{y\in Y}\beta(y,z)\mathbin{\&}\alpha(x,y)\Bigr)
\mathbin{\&}\Bigl(\bigvee_{y'\in Y'}\beta'(y',z')\mathbin{\&}\alpha'(x',y')\Bigr)\\
&\quad=(\beta\circ\alpha)(x,z)\mathbin{\&}(\beta'\circ\alpha')(x',z').
\end{aligned}
\]
The right-hand side exactly matches the pointwise definition of $(\beta\circ\alpha)*(\beta'\circ\alpha')$, which completes the proof.
\end{proof}

\begin{proposition}\label{prop:tensor-diag-compat}
Let \(f\colon\alpha\to\beta\) and \(f'\colon\alpha'\to\beta'\) be morphisms in \(\DQRel\). Then the pointwise tensor product \(f\ast f'\) defines a morphism in \(\DQRel(\alpha\ast\alpha',\beta\ast\beta')\). Furthermore, for any additional morphisms \(g\colon\beta\to\gamma\) and \(g'\colon\beta'\to\gamma'\) in \(\DQRel\), the pointwise tensor product and composition of diagonal morphisms commute in the sense that
    \[
    (g\ast g')\odt (f\ast f')=(g\odt f)\ast(g'\odt f').
    \]
\end{proposition}
    
\begin{proof}
We first verify that \(f\ast f'\) is a well-defined morphism in \(\DQRel(\alpha\ast\alpha',\beta\ast\beta')\). Since \(f\in\DQRel(\alpha,\beta)\) and \(f'\in\DQRel(\alpha',\beta')\) are diagonal morphisms, there exist \(Q\)-relations \(u,v\) and \(u',v'\) satisfying the characteristic factorisation conditions of \(\DQRel\)-morphisms:
    \[
    \beta\circ u=f=v\circ\alpha,\qquad \beta'\circ u'=f'=v'\circ\alpha'.
    \]
By virtue of Proposition \ref{prop:tensor-compat}, the tensor product operation preserves such relational factorisations, yielding
    \[
    (\beta\ast\beta')\circ(u\ast u')=f\ast f'=(v\ast v')\circ(\alpha\ast\alpha').
    \]
This two-sided factorisation property implies \(f\ast f'\in\DQRel(\alpha\ast\alpha',\beta\ast\beta')\).
    
We now prove the commutative identity relating tensor products and diagonal composition. By the definition of the diagonal composition \(\odt\) on \(\DQRel\), we have
    \[
    (g\ast g')\odt (f\ast f')=(g\ast g')\circ\bigl((\beta\ast\beta')\searrow (f\ast f')\bigr).
    \]
    Since \(g\colon\beta\to\gamma\) and \(g'\colon\beta'\to\gamma'\) are morphisms in \(\DQRel\), each diagonal morphism admits the standard decomposition \(g=(g\swarrow\beta)\circ\beta\) and \(g'=(g'\swarrow\beta')\circ\beta'\). Applying Proposition \ref{prop:tensor-compat}, the tensor product distributes over relational composition, so that
    \[
    g\ast g'=\big((g\swarrow\beta)\circ\beta\big)\ast\big((g'\swarrow\beta')\circ\beta'\big)=\big((g\swarrow\beta)\ast(g'\swarrow\beta')\big)\circ(\beta\ast\beta').
    \]
    Substituting this decomposition into the previous formula yields the following equational derivation:
    \begin{align*}
    (g\ast g')\odt (f\ast f')
    &=\big((g\swarrow\beta)\ast(g'\swarrow\beta')\big)\circ(\beta\ast\beta')\circ\bigl((\beta\ast\beta')\searrow (f\ast f')\bigr)\\
    &=\big((g\swarrow\beta)\ast(g'\swarrow\beta')\big)\circ(f\ast f')  \\
    &=\big((g\swarrow\beta)\circ f\big)\ast\big((g'\swarrow\beta')\circ f'\big)\\
    &=(g\odt f)\ast(g'\odt f').
    \end{align*}
The second step holds due to the fact $f\ast f'\in\DQRel(\beta\ast\beta',\gamma\ast\gamma')$.   The third step follows directly from Proposition \ref{prop:tensor-compat}. The final equality is exactly the definition of $\odt$.
\end{proof}

As an immediate consequence of the above interchange law, the pointwise tensor product endows the category $\DQRel$ with a symmetric monoidal structure.

\begin{corollary}\label{cor:d-monoidal}
The category $(\DQRel,*,\stk)$ is a symmetric monoidal category.
\end{corollary}

We next introduce two important classes of well-behaved $Q$-relations, namely regular and idempotent $Q$-relations, which form full subcategories closed under the pointwise tensor product. Recall that, in any category $\mathbf{C}$, a morphism $f:A\to B$ is \emph{regular} if there exists a morphism $g:B\to A$ such that $f=f\circ g\circ f$. An endomorphism $f:A\to A$ is \emph{idempotent} if $f\circ f=f$.

In the specific setting of $\QRel$, for every regular $Q$-relation $\alpha:X\nto Y$, the associated $Q$-relation $\overleftarrow{\alpha}:=(\alpha\searrow\alpha)\swarrow\alpha$ is the greatest $Q$-relation satisfying $\alpha\circ\beta\circ\alpha\leq\alpha$. Accordingly, $\alpha$ is regular if and only if $\alpha=\alpha\circ\overleftarrow{\alpha}\circ\alpha$. We denote by \(\DQRel_{\mathrm{reg}}\) the full subcategory of \(\DQRel\) consisting of all regular \(Q\)-relations. Since every idempotent \(Q\)-relation is regular, \(\DQRel_{\mathrm{reg}}\) contains the full subcategory \(\DQRel_{\mathrm{idm}}\) of all idempotent \(Q\)-relations. A \(Q\)-relation \(\alpha\) is regular (respectively, idempotent) precisely when \(\alpha^\op\) is regular (respectively, idempotent). Hence \(\DQRel_{\mathrm{reg}}\) and \(\DQRel_{\mathrm{idm}}\) are closed under the opposite operation \((-)^\op\) on \(\DQRel\); consequently, both are self-dual.

\begin{proposition}[\cite{LaiShen2018}]\label{prop:reg-idm-equiv}
Let $\alpha\colon X\nto Y$ be a regular $Q$-relation. Then the $Q$-relation $\overleftarrow{\alpha}\circ\alpha$ is idempotent and isomorphic to $\alpha$ in $\DQRel$. Consequently, $\DQRel_{\mathrm{reg}}$ is categorically equivalent to its full subcategory $\DQRel_{\mathrm{idm}}$.
\end{proposition}

For morphisms between idempotent $Q$-relations, we establish a simplified intrinsic characterization of diagonal morphisms in $\DQRel$.

\begin{proposition}\label{prop:idm-morph}
Let $\alpha$ and $\beta$ be idempotent $Q$-relations. For any $Q$-relation $f$,
\[
f\in\DQRel(\alpha,\beta)
\iff f=\beta\circ f\circ\alpha.
\]
\end{proposition}

\begin{proof}
The sufficiency is trivial. For necessity, assume $f\in\DQRel(\alpha,\beta)$. By the definition of diagonal morphisms, $(f\swarrow\alpha)\circ\alpha=f=\beta\circ(\beta\searrow f)$. Using the idempotency of $\beta$, we obtain
\[
\beta\circ f=\beta\circ\bigl(\beta\circ(\beta\searrow f)\bigr)
=\beta\circ(\beta\searrow f)=f.
\]
Similarly, the idempotency of $\alpha$ implies
\[
f\circ\alpha=\bigl((f\swarrow\alpha)\circ\alpha\bigr)\circ\alpha
=(f\swarrow\alpha)\circ\alpha=f.
\]
Combining the two identities yields $\beta\circ f\circ\alpha=f$.
\end{proof}

We further prove that regularity and idempotency of $Q$-relations are preserved under pointwise tensor products, which guarantees that the corresponding subcategories are monoidal subcategories of $\DQRel$.

\begin{proposition}\label{prop:tensor-regular}
The pointwise tensor product of two regular (respectively, idempotent) $Q$-relations is regular (respectively, idempotent).
\end{proposition}

\begin{proof}
First, let $\alpha$ and $\beta$ be regular $Q$-relations, so that $\alpha=\alpha\circ\overleftarrow{\alpha}\circ\alpha$ and $\beta=\beta\circ\overleftarrow{\beta}\circ\beta$. Applying Proposition~\ref{prop:tensor-compat} twice, we have
\[
\alpha*\beta
=(\alpha\circ\overleftarrow{\alpha}\circ\alpha)*(\beta\circ\overleftarrow{\beta}\circ\beta)
=(\alpha*\beta)\circ(\overleftarrow{\alpha}*\overleftarrow{\beta})\circ(\alpha*\beta),
\]
which verifies the regularity of $\alpha*\beta$.

Next, suppose $\alpha$ and $\beta$ are idempotent. Again by Proposition~\ref{prop:tensor-compat},
\[
(\alpha*\beta)\circ(\alpha*\beta)
=(\alpha\circ\alpha)*(\beta\circ\beta)=\alpha*\beta,
\]
so $\alpha*\beta$ is idempotent. This completes the proof.
\end{proof}

We proceed to establish the symmetric monoidal closed structure of $\DQRel_{\mathrm{idm}}$, which provides a quantale-enriched counterpart of a classical result (see \cite[Proposition 2]{RosebrughWood1994}).

\begin{proposition}\label{prop:idm-closed}
The category $(\DQRel_{\mathrm{idm}},*,\stk)$ is symmetric monoidal closed, with the closed structure induced by the adjunction $(-)*\alpha^\op\dashv\alpha*{(-)}$.
\end{proposition}

\begin{proof}
It suffices to establish a natural bijection
\[
\DQRel_{\mathrm{idm}}(\beta*\alpha^{\op},\gamma)
\cong
\DQRel_{\mathrm{idm}}(\beta,\alpha*\gamma)
\]
for all idempotent $Q$-relations $\alpha\colon X\nto X$, $\beta\colon Y\nto Y$, and $\gamma\colon Z\nto Z$.

Given $f\in\DQRel_{\mathrm{idm}}(\beta*\alpha^{\op},\gamma)$, define the $Q$-relation $\widetilde{f}\colon Y\nto X\times Z$ by $\widetilde{f}(y,(x,z))=f((y,x),z)$ for all $x\in X$, $y\in Y$, $z\in Z$. By Proposition~\ref{prop:idm-morph}, it suffices to show that $\widetilde{f}=(\alpha*\gamma)\circ\widetilde{f}\circ\beta$.

Direct pointwise computation yields
\[
\begin{aligned}
&\bigl((\alpha*\gamma)\circ\widetilde{f}\circ\beta\bigr)(y,(x,z))\\
&\quad=\bigvee_{y'\in Y}\bigvee_{x'\in X,\,z'\in Z}
(\alpha*\gamma)((x',z'),(x,z))
\mathbin{\&}\widetilde{f}(y',(x',z'))\mathbin{\&}\beta(y,y')\\
&\quad=\bigvee_{x'\in X,\,y'\in Y,\,z'\in Z}
\alpha^{\op}(x,x')\mathbin{\&}\gamma(z',z)\mathbin{\&}f((y',x'),z')\mathbin{\&}\beta(y,y')\\
&\quad=\bigvee_{y'\in Y,\,x'\in X}
\biggl(\bigvee_{z'\in Z}\gamma(z',z)\mathbin{\&}f((y',x'),z')\biggr)
\mathbin{\&}(\beta*\alpha^{\op})((y,x),(y',x'))\\
&\quad=(\gamma\circ f\circ(\beta*\alpha^{\op}))((y,x),z)\\
&\quad=f((y,x),z)\\
&\quad=\widetilde{f}(y,(x,z)),
\end{aligned}
\]
where the penultimate equality follows from the hypothesis $f\in\DQRel_{\mathrm{idm}}(\beta*\alpha^{\op},\gamma)$. 
\end{proof}

We now recall the definition of constructively completely distributive $Q$-categories from \cite{LaiShen2018,Stubbe2007b}, which constitute the essential image of the functor $\frD$ associated with regular $Q$-relations. A $Q$-category $\mathbb{A}$ is \emph{constructively completely distributive} (ccd for short) if there exists a chain of adjoint $Q$-functors
\[
T_{\mathbb{A}}\dashv S_{\mathbb{A}}\dashv Y_{\mathbb{A}}\colon
\mathbb{A}\to\mathcal{P}\mathbb{A}.
\]
This notion generalizes constructive complete distributivity of classical complete lattices \cite{Wood1990,KenneyWood2010,RosebrughWood1994} to the quantale-enriched categorical setting.

Let $\QSup_{\mathrm{ccd}}$ denote the full subcategory of $\QSup$ consisting of all ccd cocomplete $Q$-categories. Existing results in \cite{LaiShen2018,Stubbe2007b} confirm that the cocomplete $Q$-category $\mathfrak{D}\alpha$ associated with any regular $Q$-relation $\alpha$ is ccd. Conversely, every ccd $Q$-category can be realized as $\mathfrak{D}\theta$ for some idempotent $Q$-relation $\theta$ \cite{LaiShen2018,Stubbe2007b}.

Explicitly, for a ccd $Q$-category $\mathbb{A}$, its \emph{totally below distributor} \cite{Stubbe2007b} is defined as
\[
\theta_{\mathbb{A}}:=(T_{\mathbb{A}})^{\sharp}\circ(Y_{\mathbb{A}})_{\natural},
\]
where $(T_{\mathbb{A}})^{\sharp}$ and $(Y_{\mathbb{A}})_{\natural}$ stand for the cograph of $T_{\mathbb{A}}$ and the graph of $Y_{\mathbb{A}}$, respectively. By construction, $\theta_{\mathbb{A}}\colon\mathbb{A}\nto\mathbb{A}$ is a $Q$-distributor with the pointwise expression
\[
\theta_{\mathbb{A}}(x,y)=\mathcal{P}\mathbb{A}(Y_{\mathbb{A}}(x),T_{\mathbb{A}}(y))=T_{\mathbb{A}}(y)(x).
\]
This distributor satisfies $\theta_{\mathbb{A}}\leq\mathbb{A}$ and $\theta_{\mathbb{A}}\circ\theta_{\mathbb{A}}=\theta_{\mathbb{A}}$, hence is idempotent. Most importantly, it admits a canonical isomorphism $\mathfrak{D}\theta_{\mathbb{A}}\cong\mathbb{A}$.

\begin{proposition}[\cite{Stubbe2005b,Stubbe2007b}]\label{prop:idm-hom}
Let \(\alpha\) and \(\beta\) be idempotent \(Q\)-relations. Then there is a natural isomorphism
$
\QSup(\mathfrak{D}\alpha,\mathfrak{D}\beta)
\cong \DQRel(\beta,\alpha).
$
Moreover, \(\mathfrak{D}\alpha\cong\mathfrak{D}\beta\) in \(Q\text{-}\mathbf{Cat}\) if and only if \(\alpha\cong\beta\) in \(\DQRel\).
\end{proposition}

The above proposition establishes a direct categorical equivalence between idempotent $Q$-relations and ccd cocomplete $Q$-categories.

\begin{corollary}\label{cor:D-equivalence}
Restricting the functor $\mathfrak{D}$ yields a category equivalence between $\DQRel_{\mathrm{idm}}^\op$ and $\QSup_{\mathrm{ccd}}$.
\end{corollary}

The following commutative diagram summarizes the hierarchical relationships among the aforementioned categories:
\[
\begin{tikzcd}
\DQRel_{\mathrm{idm}}^{\op}
\arrow[>->]{rr}{} \arrow{d}[swap]{\simeq} &&
\DQRel^{\op} \arrow{d}{\mathfrak{D}} \\
\QSup_{\mathrm{ccd}}\arrow[>->]{rr}{}&&
 \QSup
\end{tikzcd}
\]

By Proposition~\ref{prop:idm-closed}, \(\DQRel_{\mathrm{idm}}\) is symmetric monoidal closed. Since \(\DQRel_{\mathrm{idm}}\) is self-dual, its opposite category \(\DQRel_{\mathrm{idm}}^\op\) is symmetric monoidal closed as well. The equivalence functor $\mathfrak{D}$ thus transports this monoidal closed structure to $\QSup_{\mathrm{ccd}}$. We prove in the following that the transported tensor product coincides exactly with the standard tensor product $\otimes$ inherited from the ambient category $\QSup$. It is a quantale-enriched counterpart of the classical result on constructively completely distributive lattices (see \cite[Theorem 6.3]{KenneyWood2010}).

\begin{theorem}\label{thm:main}
Let $\alpha$ and $\beta$ be regular $Q$-relations. There exists a natural isomorphism
\[
\mathfrak{D}\alpha\otimes\mathfrak{D}\beta\cong\mathfrak{D}(\alpha*\beta)
\]
in the category $\QSup$. In particular, the tensor product of two constructively completely distributive cocomplete $Q$-categories remains constructively completely distributive.
\end{theorem}

\begin{proof}
Let $\bbA=\frD\alpha$ and $\bbB=\frD\beta$, so that both $\bbA$ and $\bbB$ are ccd cocomplete $Q$-categories. As noted above, every ccd $Q$-category is isomorphic to the $\mathfrak{D}$-image of its associated idempotent totally-below distributor. Accordingly, there exist idempotent distributors $\theta_{\bbA}$ and $\theta_{\bbB}$ (the totally below relations of $\bbA$ and $\bbB$) such that $\frD\theta_{\bbA}\cong\bbA$ and $\frD\theta_{\bbB}\cong\bbB$. It suffices to verify $\frD(\theta_{\bbA}*\theta_{\bbB})\cong\bbA\otimes\bbB$.

By the symmetric monoidal closed structure of $\QSup$, we have the canonical isomorphism
\[
\bbA\otimes\bbB\cong(\bbA\otimes\bbB)\otimes \bbQ\cong\QSup(\bbA\otimes\bbB,\bbQ^{\op})^{\op},
\]
which identifies the objects of $\bbA\otimes\bbB$ with $\QSup$-bimorphisms $\bbA\times\bbB\to\bbQ^{\op}$. We show that such bimorphisms bijectively correspond to the elements of $\frD(\theta_{\bbA}*\theta_{\bbB})$.

Set $\theta=\theta_{\bbA}*\theta_{\bbB}$. For any $\QSup$-bimorphism $u\colon\bbA\times\bbB\to \bbQ^{\op}$ and all $x\in\bbA_0$, $y\in\bbB_0$, the ccd property of $\bbA$ and $\bbB$ guarantees $x=S_{\bbA}(T_{\bbA}x)$ and $y=S_{\bbB}(T_{\bbB}y)$. Expanding these adjoint representations gives
\begin{align*}
u(x,y)&=u\bigl(S_{\bbA}(T_{\bbA}x),S_{\bbB}(T_{\bbB}y)\bigr)\\
&=u\biggl(\bigvee_{x'\in\bbA_0}T_{\bbA}x(x')\cdot x',\,
\bigvee_{y'\in\bbB_0}T_{\bbB}y(y')\cdot y'\biggr)\\
&=\bigwedge_{x'\in\bbA_0}\theta_{\bbA}(x',x)\ra
\biggl(\bigwedge_{y'\in\bbB_0}\theta_{\bbB}(y',y)\ra u(x',y')\biggr)\\
&=\bigwedge_{x'\in\bbA_0,\,y'\in\bbB_0}
\bigl(\theta_{\bbA}(x',x)\mathbin{\&}\theta_{\bbB}(y',y)\bigr)\ra u(x',y')\\
&=(u\swarrow\theta)(x,y).
\end{align*}
The third equality follows from the bimorphism property of $u$: the hom functor on $\bbQ^{\op}$ interchanges joins and meets and converts the scalar action $q\cdot z$ into scalar implication $q\ra z$. This computation implies $u=u\swarrow\theta$, so $u$ lies in the image of the lifting operation $(-)\swarrow\theta\colon \QRel(\bbA_0\times\bbB_0,\sigt)\to\QRel(\bbA_0\times\bbB_0,\sigt)$.

Conversely, for any $Q$-relation $f\colon\bbA_0\times\bbB_0\nto\sigt$, define $u=f\swarrow\theta\colon\bbA_0\times\bbB_0\nto\sigt$. By definition of the lifting operator,
\[
(f\swarrow\theta)(x,y)=f\swarrow(T_\bbA(x)\ast T_\bbB(y)).
\]
This induces the commutative diagram
\[
\begin{tikzcd}
\bbA\times\bbB
\arrow{rr}{u=f\swarrow\theta} \arrow{dr}[swap]{T_{\bbA}(-)*T_{\bbB}(-)} &&
\bbQ^{\op} \\
&\mathcal{P}\delta_{\bbA_0\times\bbB_0}
\arrow{ur}[swap]{f\swarrow(-)} &
\end{tikzcd}
\]

Here, $T_{\bbA}(-)*T_{\bbB}(-)$ is a $\QSup$-bimorphism as the pointwise tensor product of adjoint functors, while $f\swarrow(-)$ is a $\QSup$-morphism. Their composite $u$ is therefore a $\QSup$-bimorphism.

The above arguments demonstrate that the set $\mathbf{bi}\text{-}\QSup(\bbA\times\bbB,\bbQ^\op)^\op$ of all $\QSup$-bimorphisms from $\bbA\times\bbB$ to $\bbQ^\op$ coincides exactly with the image of the $Q$-functor $(-)\swarrow\theta$. Since $(-)\swarrow\theta$ is right adjoint to $(-)\circ\theta$, its image is isomorphic to $\frD\theta$. We thus obtain the natural isomorphism
\[
\frD(\theta_{\bbA}*\theta_{\bbB})=\frD\theta
\cong\mathbf{bi}\text{-}\QSup(\bbA\times\bbB,\bbQ^{\op})^{\op}
\cong\bbA\otimes\bbB.
\]

By Proposition \ref{prop:idm-hom}, $\theta_\bbA\cong\alpha$ and $\theta_\bbB\cong\beta$ in $\DQRel_\mathrm{reg}$, which yields
\[
\frD\alpha\otimes\frD\beta=\bbA\otimes\bbB\cong\frD(\theta_\bbA\ast\theta_\bbB)\cong\frD(\alpha*\beta).
\]
By Proposition~\ref{prop:tensor-regular}, $\alpha*\beta$ is regular, so $\frD(\alpha*\beta)$ is ccd. Consequently, the tensor product of two ccd cocomplete $Q$-categories remains ccd.
\end{proof}

\begin{remark}
Let \(Q\) be a Girard quantale. By the complementation isomorphism between \(\DQRel\) and \(\BQRel\) (Lemma~\ref{lem:girard-b}), we have \(\frB\alpha=\frD(\neg\alpha)\) for every \(Q\)-relation \(\alpha\). Moreover, \(\frD(\neg\alpha)\) is constructively completely distributive (ccd) if and only if \(\neg\alpha\) is regular (see \cite{LaiShen2018}). Hence, for a \(Q\)-relation \(\alpha\colon X\nto Y\), the formal concept lattice \(\frB\alpha\) is ccd if and only if \(\neg\alpha\) is regular.

If both \(\neg\alpha\) and \(\neg\beta\) are regular, then
\[
\mathfrak{B}\alpha\otimes\mathfrak{B}\beta
\cong\mathfrak{D}(\neg\alpha)\otimes\mathfrak{D}(\neg\beta)
\cong\mathfrak{D}((\neg\alpha)*(\neg\beta))
= \frB(\neg((\neg\alpha)*(\neg\beta))).
\]
In particular, in the classical case \(Q=\{0,1\}\), the formal concept lattice \(\frB(X,Y,\alpha)\) is ccd if and only if \(\neg\alpha\) is regular. For contexts \((X,Y,\alpha)\) and \((X',Y',\beta)\),
\[
((x,y),(x',y'))\in \neg((\neg\alpha)*(\neg\beta)) 
\iff (x,y)\in\alpha \text{ or } (x',y')\in\beta.
\]
Thus \((X\times X',Y\times Y',\neg((\neg\alpha)*(\neg\beta)))\) is the direct product of \((X,Y,\alpha)\) and \((X',Y',\beta)\) (see \cite[Section 1.3, Definition 32]{Ganter2024}). Therefore, Theorem \ref{thm:main} recovers Theorem 33 in \cite[Section 4.4]{Ganter2024}.
\end{remark}

We further prove that the exponential object of two ccd $Q$-categories $\bbA$ and $\bbB$ in $\QSup$ is also ccd. Proposition~\ref{prop:idm-closed} provides the natural bijection
\[
\mathbf{D}(\QRel)(\theta_\bbB\ast\theta_{\bbA}^{\op},\star_{\mathrm{k}})
\cong\mathbf{D}(\QRel)(\theta_{\mathbb{B}},\star_{\mathrm{k}}*\theta_{\mathbb{A}})\cong \mathbf{D}(\QRel)(\theta_{\mathbb{B}},\theta_{\mathbb{A}}),
\]
which implies
\[
\mathfrak{D}(\theta_\bbB\ast\theta_{\bbA}^{\op})
\cong\QSup(\mathfrak{D}\theta_{\mathbb{A}},\mathfrak{D}\theta_{\mathbb{B}})
\cong\QSup(\mathbb{A},\mathbb{B}).
\]

Accordingly, the internal hom $\mathbb{B}^{\mathbb{A}}$ of $\bbA$ and $\bbB$ in $\QSup_{\mathrm{ccd}}$ is realized as $\mathfrak{D}(\theta_\bbB\ast\theta_{\bbA}^{\op})$. Since $\theta_\bbB\ast\theta_{\bbA}^{\op}$ is idempotent by Proposition~\ref{prop:tensor-regular}, we obtain the following result.

\begin{proposition}\label{prop:internal-hom-ccd}
Let $\mathbb{A}$ and $\mathbb{B}$ be ccd $Q$-categories. Then their internal hom $\mathbb{B}^{\mathbb{A}}$ in $\QSup$ is also ccd.
\end{proposition}

As a direct consequence, the subcategory of ccd $Q$-categories inherits the full monoidal closed structure of $\QSup$.

\begin{corollary}\label{cor:ccd-monoidal}
The full subcategory $\QSup_{\mathrm{ccd}}$ is closed under the standard tensor product $\otimes$ and exponential objects $\bbB^\bbA$ of $\QSup$. Hence, $(\QSup_\mathrm{ccd},\otimes,\bbQ)$ forms a symmetric monoidal closed category.
\end{corollary}

\begin{proof}
It suffices to verify that the unit object $\bbQ$ is ccd, which follows directly from the isomorphism $\bbQ\cong\mathcal{P}\star_\mathrm{k}=\frD\star_\mathrm{k}$.
\end{proof}

We conclude this section with an explicit counterexample demonstrating that the regularity hypothesis in Theorem~\ref{thm:main} is indispensable and cannot be omitted.

\begin{example}
Let $Q=([0,1],\leq,\&)$ denote the standard Łukasiewicz quantale, with monoidal and implication operations defined by
\[
x\mathbin{\&} y=\max\{x+y-1,0\},\quad x\ra y=\min\{1-x+y,1\}.
\]

Consider the singleton-valued $Q$-relation
\[
\alpha\colon\sigt\nto\sigt,\qquad \alpha(\star,\star)=0.5.
\]
Direct verification confirms that $\alpha$ is non-regular. We first characterize the associated $Q$-category $\mathfrak{D}\alpha$. A $Q$-relation $\lambda\colon\sigt\nto\sigt$ belongs to $\mathfrak{D}\alpha$ if and only if it satisfies the fixed-point condition
\[
\lambda(\star,\star)=\bigl(\alpha(\star,\star)\ra\lambda(\star,\star)\bigr)\,\&\, \alpha(\star,\star)
=\min\bigl\{\lambda(\star,\star),\,0.5\bigr\}.
\]
This forces $\lambda(\star,\star)\in[0,0.5]$, so we identify $\mathfrak{D}\alpha$ with the interval $[0,0.5]$. Its enriched hom-operation and scalar $Q$-action are given by
$\mathfrak{D}\alpha(p,q)=p\ra q=\min\{1-p+q,\,1\}$,
 $q\cdot p=q\mathbin{\&} p=\max\{q+p-1,\,0\}$ and $p^{q}=q\ra p=\min\{1-q+p,\,0.5\}$.

The objects of the tensor product $\mathfrak{D}\alpha\otimes\mathfrak{D}\alpha$ correspond bijectively to left-adjoint $Q$-functors $F\colon\mathfrak{D}\alpha\to(\mathfrak{D}\alpha)^{\mathrm{op}}$, which are uniquely determined by two conditions:
\begin{enumerate}[label=(\arabic*)]
\item $F(q\mathbin{\&} x)=q\ra F(x)$ for all $q\in Q$ and $x\in[0,0.5]$;
\item $F(\sup X)=\inf F(X)$ for all subsets $X\subseteq[0,0.5]$.
\end{enumerate}

Every such functor admits a piecewise explicit form
\[
F_a(t)=
\begin{cases}
0.5, & 0\leq t\leq a,\\
a+0.5-t, & a<t\leq 0.5,
\end{cases}
\]
parameterized by $a\in[0,0.5]$. This yields the object isomorphism
\[
(\mathfrak{D}\alpha\otimes\mathfrak{D}\alpha)_0\cong\{F_a\mid 0\leq a\leq 0.5\}.
\]

On the other hand, the pointwise self-tensor product of $\alpha$ satisfies
\[
(\alpha*\alpha)(\star,\star)=\alpha(\star,\star)\mathbin{\&}\alpha(\star,\star)
=\max\{0.5+0.5-1,\,0\}=0.
\]
A $Q$-relation $\lambda\colon\sigt\nto\sigt$ lies in $\mathfrak{D}(\alpha*\alpha)$ if and only if $\lambda=(\lambda\swarrow(\alpha*\alpha))\circ(\alpha*\alpha)$. Since $\alpha*\alpha$ is the zero $Q$-relation, the only feasible solution is $\lambda(\star,\star)=0$, so that $\mathfrak{D}(\alpha*\alpha)\cong\{0\}$.

We therefore obtain a strict non-isomorphism
\[
\mathfrak{D}\alpha\otimes\mathfrak{D}\alpha\not\cong\mathfrak{D}(\alpha*\alpha),
\]
which demonstrates that the tensor product isomorphism established in Theorem~\ref{thm:main} fails for non-regular $Q$-relations, validating the necessity of the regularity assumption.
\end{example}

 

\bibliographystyle{plain}
\bibliography{ref}
\end{document}